\documentclass[reqno, 12pt]{amsart}
\usepackage{a4wide, amsfonts, amssymb, amsmath, amsthm, mathrsfs, enumerate, enumitem, setspace, url}
\usepackage[utf8]{inputenc}
\usepackage{tikz-cd} 
	\usetikzlibrary{cd}
\usepackage{color}
\usepackage{hyperref} 
	\definecolor{darkred}{rgb}{0.5,0,0}
	\definecolor{darkgreen}{rgb}{0,0.5,0}
	\definecolor{darkblue}{rgb}{0,0,0.5}
	\hypersetup{linkcolor=darkblue,filecolor=darkgreen,urlcolor=darkred,citecolor=darkblue}

\usepackage{comment}

\usepackage[T2A,T1]{fontenc}
\DeclareSymbolFont{cyrillic}{T2A}{cmr}{m}{n}
\DeclareMathSymbol{\Sha}{\mathalpha}{cyrillic}{216}

\usepackage[toc,page]{appendix}
\newcommand{\stoptocwriting}{%
	\addtocontents{toc}{\protect\setcounter{tocdepth}{-5}}
	}
\newcommand{\resumetocwriting}{%
	\addtocontents{toc}{\protect\setcounter{tocdepth}{\arabic{tocdepth}}}
	}

\makeatletter
\newcommand{\leqmode}{\tagsleft@true}
\newcommand{\reqmode}{\tagsleft@false}
\makeatother

\theoremstyle{plain}
\newtheorem{theorem}{Theorem}[section]
\newtheorem*{theorem*}{Theorem}
\newtheorem{proposition}[theorem]{Proposition}
\newtheorem{lemma}[theorem]{Lemma}

\theoremstyle{remark}
\newtheorem{remark}[theorem]{Remark}
\newtheorem*{acknowledgements}{Acknowledgements}

\theoremstyle{definition}

\numberwithin{equation}{section}

\renewcommand{\(}{\left(}
\renewcommand{\)}{\right)}

\newcommand{\ZZ}{\mathbb{Z}}
\newcommand{\QQ}{\mathbb{Q}}
\newcommand{\Zp}{\mathbb{Z}_p}

\newcommand{\Qp}{\mathbb{Q}_p}

\newcommand{\RR}{\mathbb{R}}

\newcommand{\FF}{\mathbb{F}}
\newcommand{\Fp}{\mathbb{F}_p}

\renewcommand{\AA}{\mathbb{A}}
\newcommand{\PP}{\mathbb{P}}
\newcommand{\GG}{\mathbb{G}}

\newcommand{\bfa}{\mathbf{a}} \newcommand{\bfA}{\mathbf{A}}

\newcommand{\bfx}{\mathbf{x}}

\newcommand{\bfalpha}{\boldsymbol{\alpha}}

\newcommand{\sA}{\mathcal{A}}

\newcommand{\sF}{\mathcal{F}}
\newcommand{\sO}{\mathcal{O}}

\DeclareMathOperator{\HH}{H}

\DeclareMathOperator{\Br}{Br}
\DeclareMathOperator{\inv}{inv}
\DeclareMathOperator{\ev}{ev}
\DeclareMathOperator{\Res}{Res}

\DeclareMathOperator{\id}{id}

\DeclareMathOperator{\val3}{v_3}
\DeclareMathOperator{\valp}{v_{\it p}}

\renewcommand{\epsilon}{\varepsilon}

\newcommand{\nequiv}{\not\equiv}

\begin{document}
\onehalfspacing
\title[Quartic del Pezzo surfaces]{Rational points on del Pezzo surfaces of degree four}
\author{Vladimir Mitankin}
	\address{Max Planck Institut f\"{u}r Mathematik \\
		Vivatsgasse 7 \\
		53111 Bonn \\
		Germany
	}
	\email{vmitankin@mpim-bonn.mpg.de}
\author{Cec\'ilia Salgado}
	\address{Instituto de Matem\'{a}tica, UFRJ \\
		Centro de Tecnologia - Bloco C, Cidade Universit\'{a}ria \\
		Athos da Silveira Ramos 149 \\
		Ilha do Fund\~{a}o, Rio de Janeiro, RJ \\
		21941--909 \\
		Brazil and \newline
		Max Planck Institut f\"{u}r Mathematik \\
		Vivatsgasse 7 \\
		53111 Bonn \\
		Germany
	}
	\email{odaglas.ailicec@gmail.com}
\date{\today}
\thanks{2020 {\em Mathematics Subject Classification} 
	 14G12 (primary), 11G35, 11D09, 14D10, 14G05 (secondary).
}

\begin{abstract}
  We study the distribution of the Brauer group and the frequency of the Brauer--Manin obstruction to the Hasse principle and weak approximation in a family of smooth del Pezzo surfaces of degree four over the rationals.
\end{abstract}  
\maketitle
\tableofcontents
\setcounter{tocdepth}{1}

\section{Introduction}
A del Pezzo surface of degree four $X$ over $\QQ$ is a smooth projective surface in $\mathbb{P}^4$ given by the complete intersection of two quadrics defined over $\QQ$. Such surfaces have been the object of study of several papers throughout the last half century. A prominent reason for that is the fact that they provide the simplest example of failure of the Hasse Principle for surfaces. We say that an algebraic variety defined over $\QQ$ \emph{satisfies the Hasse principle} if it has a $\QQ$-point whenever it has a real point and and a point in each field of $p$-adic numbers $\Qp$. The simplest class of surfaces, namely those with Kodaira dimension $-\infty$, is formed by rational and ruled surfaces. The arithmetic of the latter is determined by that of del Pezzo and conic bundle surfaces. Del Pezzo surfaces see their level of arithmetic and geometric complexity increase inversely to their degrees and those of degree at least 5 admitting a rational point are always $\QQ$-rational. Hence quartic del Pezzo surfaces form the first class for which interesting arithmetic phenomena, as for instance failures of the Hasse Principle, can occur. They are the object of study of this paper.

We say that a variety $X$ over $\QQ$ \emph{satisfies weak approximation} if the set of its rational points $X(\QQ)$ is dense in the ad\`eles $X(\bfA_\QQ)$ of $X$, that is $X(\QQ)$ is dense in $\prod_{p \in S} X(\Qp)$ for any finite set $S$ of places of $\QQ$. A conjecture of Colliot-Th\'{e}l\`{e}ne and Sansuc \cite{CTS80} predicts that all failures of the Hasse principle and weak approximation are explained by the Brauer--Manin obstruction. This is a cohomological obstruction developed by Manin \cite{Man74} which exploits the fact that for a smooth, geometrically irreducible variety $X$ over $\QQ$ there is pairing between $X(\bfA_\QQ)$ and the Brauer group $\Br X = \HH_{\text{\'{e}t}}^2(X,\mathbb{G}_m)$ of $X$. Manin showed that $X(\QQ)$ lies inside the left kernel $X(\bfA_\QQ)^{\Br X}$ of this pairing. A \emph{Brauer--Manin obstruction to the Hasse principle} is then present if $X(\bfA_\QQ) \neq \emptyset$ but $X(\bfA_\QQ)^{\Br X} = \emptyset$. On the other hand, there is a \emph{Brauer--Manin obstruction to weak approximation} if $X(\bfA_\QQ)^{\Br X} \neq X(\bfA_\QQ)$. Such  obstructions may occur only if $\Br X / \Br \QQ$ is non-trivial. 

Colliot-Th\'el\`ene and Sansuc's conjecture is established, for a general del Pezzo surface of degree four, under Schinzel’s hypothesis and the finiteness of Tate--Shafarevich groups of elliptic curves by Wittenberg \cite[Thm.~3.36.]{Wit07} when $\Br X = \Br \QQ$ and by V\'arilly-Alvarado and Viray \cite[Thm.~1.5.]{VAV14} when $X$ is of BSD type. The latter corresponds to the complete intersection in $\PP^4$ of the following two quadrics
\[
	\begin{split}
		cx_3 x_4 &= x_2^2 - \varepsilon x_0^2, \\
		(x_3 + x_4)(ax_3 + bx_4) &= x_2^2 - \varepsilon x_1^2,
	\end{split}
\]
where $a, b, c \in \QQ^*$ and $\varepsilon \in \QQ \setminus \QQ^{*2}$ with $(a - b)(a^2 + b^2 + c^2 - ab - ac - bc) \neq 0$ and
$ab, \varepsilon (a^2 + b^2 + c^2 - ab - ac - bc) \notin \QQ^{*2}$.

Quartic del Pezzo surfaces of BSD type with an adelic point always have $\Br X/\Br \QQ \simeq \ZZ/2\ZZ$ . In this setting, Jahnel and Schindler \cite{JS17} have shown that all counter-examples to the Hasse principle form a Zariski dense set in the moduli scheme of all del Pezzo surfaces of degree four.

In this paper we consider a different family of del Pezzo surfaces of degree four given as follows. Let $\bfa = (a_0, \dots, a_4)$ be a quintuple with coprime integer coordinates, we denote the set of such vectors by $\ZZ_{\text{prim}}^5$. Then define $X_\bfa \subset \PP_\QQ^4$ by the complete intersection
\begin{equation}
	\label{eq:dP4 main}
	\begin{split}
		x_0x_1 - x_2x_3 = 0, \\
		a_0x_0^2 + a_1x_1^2 + a_2x_2^2 + a_3x_3^2 + a_4x_4^2 = 0. 
	\end{split}
\end{equation}
We can assume that $a_0, \dots, a_4$ have no factor in common without any loss of generality, otherwise divide the second equation by that factor. Such $X_\bfa$ are smooth if and only if $(a_0a_1 - a_2a_3)\prod_{i = 0}^4 a_i \neq 0$. Here we are interested in the family of all smooth quartic del Pezzo surfaces given by \eqref{eq:dP4 main}, that is
\[	
	\sF = \{X_\bfa \mbox{ as in } \eqref{eq:dP4 main} \ : \ \bfa \in \ZZ_{\text{prim}}^5 \mbox{ and } (a_0a_1 - a_2a_3)\prod_{i = 0}^4 a_i \neq 0\}. 
\]

There are numerous reasons behind our choice of this family. Firstly, different from the BSD type family, the Brauer group does witness a variation as $\bfa$ runs through $\ZZ_{\text{prim}}^5$ which makes interesting the problem of studying the frequency of each possible Brauer group. Secondly, surfaces in $\sF$ admit two distinct conic bundle structures, making their geometry and hence their arithmetic considerably more tractable. Moreover, for such surfaces Colliot-Th\'{e}l\`{e}ne and Salberger have independently shown that the Brauer--Manin obstruction is the only obstruction to the Hasse principle and weak approximation \cite[Thm~2]{CT90}, \cite{Sal86}. In particular, the conjecture of Colliot-Th\'{e}l\`{e}ne and Sansuc holds unconditionally for $X_\bfa$. Finally, our surfaces can be thought of as an analogue of diagonal cubic surfaces as they also satisfy the interesting equivalence of $\QQ$-rationality and trivial Brauer group. This is shown in our forthcoming work.

We order $X_\bfa \in \sF$ with respect to the naive height function $|\bfa| = \max_{0 \le i \le 4}{|a_i|}$. Recall that since $X_\bfa$ is projective we have $X_{\bfa}(\bfA_\QQ) = \prod_{p \le \infty} X_{\bfa}(\Qp)$, where we have used $\QQ_\infty$ to denote the reals. The coprimality condition of $\ZZ_{\text{prim}}^5$ implies that the number of $X \in \sF$ of height at most $B$ is asymptotically $32B^5/\zeta(5)$, where $\zeta(s)$ is the Riemann zeta function. Our first result shows that a positive proportion of $X_\bfa \in \sF$ have points everywhere locally, i.e. $X_\bfa(\bfA_\QQ) \neq \emptyset$.

\begin{theorem}
	\label{thm:local solubility}
	We have
	\[
		\lim_{B \rightarrow \infty}\frac{\#\left\{X_\bfa \in \sF \ : \ |\bfa| \le B \mbox{ and } X_\bfa(\bfA_\QQ) \neq \emptyset \right\}}{32B^5/\zeta(5)} = \sigma_\infty \prod_p \sigma_p > 0,
	\]
	where $\sigma_p$, $\sigma_\infty$ are local densities whose values are given in Proposition~\ref{prop:local densities}.
\end{theorem}

A first natural step towards understanding the frequency of failures of the Hasse principle and weak approximation for $X_\bfa \in \sF$ is to understand how often $\Br X_{\bfa} / \Br \QQ$ is non-trivial. It is well-understood for quartic del Pezzo surfaces that when non-trivial this quotient is either $\ZZ / 2\ZZ$ or $(\ZZ/2\ZZ)^2$ \cite{Man74}, \cite{SD93}. To analyse this for any real $B \ge 1$ let
\[
	N_{\# \sA}(B) 
	= \#\{ X_{\bfa} \in \sF \ : \ |\bfa| \le B, X_\bfa(\bfA_\QQ) \neq \emptyset \mbox{ and } \Br X_{\bfa} / \Br \QQ \simeq \sA\},
\]
where $\sA$ is either the trivial group, $\ZZ/2\ZZ$ or $(\ZZ/2\ZZ)^2$. Our next result shows that $\Br X_\bfa / \Br \QQ$ is almost always of order two.

\begin{theorem} 
	\label{thm:Br}
	We have
	\[
		\begin{split}
			B^3 &\ll N_1(B) \ll B^3 (\log B)^4, \\
			N_2(B) &\sim \frac{32}{\zeta(5)}\(\sigma_\infty \prod_p \sigma_p \)B^5, \\
			N_4(B) &= \frac{60}{\pi^2}B^3 + O(B^{5/2}(\log B)^2),
		\end{split}
	\]
	as $B$ goes to infinity.
\end{theorem}

Theorem~\ref{thm:Br} implies that there are infinitely many $X_\bfa$ with $\Br X_{\bfa} / \Br \QQ$ of order four. However, Remark~\ref{rem:Br4} shows that such surfaces always have a rational point and thus all failures of the Hasse principle in $\sF$ arise when $\Br X_\bfa/\Br \QQ = \ZZ/2\ZZ$. Our next result provides an upper bound for the number of such failures and shows that they appear quite rarely in the family $\sF$.

\begin{theorem}
	\label{thm:BMO}
	We have
	\[
		\#\{X_\bfa \in \sF \ : \ |\bfa| \le B, X_\bfa(\bfA_\QQ) \neq \emptyset \mbox{ but } X_\bfa(\QQ) = \emptyset \}
		\ll B^{9/2},
	\]
	as $B$ goes to infinity.
\end{theorem}

Theorem~\ref{thm:BMO} together with Theorem~\ref{thm:local solubility} give us a better understanding of how often varieties in families have a rational point. This question has raised a significant interest lately with studies by numerous authors \cite{Bha14}, \cite{BBL16}, \cite{BB14} \cite{Lou18}, \cite{LS16}, \cite{Ser90}, \cite{Sof16}. An answer to it in complete generality seems out of reach with current techniques which makes results as in Theorem~\ref{thm:BMO} especially valuable.

The proof of Theorem~\ref{thm:BMO} yields that $100\%$ of the surfaces in $\sF$ satisfy the Hasse principle but yet fail weak approximation. This is made precise in the next theorem.

\begin{theorem}
	\label{thm:WA}
	We have
	\[
		\#\{X_\bfa \in \sF \ : \ |\bfa| \le B, X_\bfa \mbox{ satisfies weak approximation}\}
		\ll B^{9/2},
	\]
	as $B$ goes to infinity.
\end{theorem}

It follows from Theorem~\ref{thm:Br} that the quantity in Theorem~\ref{thm:WA} is $\gg B^3$, since rational surfaces satisfy weak approximation. There are general methods for proving results as in Theorems~\ref{thm:BMO} and \ref{thm:WA} developed in \cite{BBL16} and \cite{Bri18}. However, these methods would yield a bound of the shape $O(B^5/(\log B))$. While our idea closely resembles the one used in \cite{BBL16} and \cite{Bri18}, the explicit description of certain Brauer group elements here allows us to get a power saving in the upper bounds obtained in Theorems~\ref{thm:BMO} and \ref{thm:WA}. Thus Theorem~\ref{thm:BMO} and \ref{thm:WA} do not follow from the general tools.

This paper is organised as follows. In Section~\ref{sec:Brauer} we describe explicitly the two conic bundle structures on $X_\bfa$ and use them to compute $\Br X_\bfa / \Br \QQ$. Section~\ref{sec:local} is dedicated to the study of the local points and the local densities $\sigma_p$ and $\sigma_\infty$. In Section~\ref{sec:thm local} we prove Theorem~\ref{thm:local solubility}. Sections~\ref{sec:br4} and \ref{sec:thm Br} are dedicated to the proof of Theorem~\ref{thm:Br}. The proofs of Theorems~\ref{thm:BMO} and \ref{thm:WA} are contained in Section~\ref{sec:bmo}.

\stoptocwriting
\subsection*{Notation}
Throughout this paper we set $d = a_0a_1 - a_2a_3$. For a field $k$ and a variety $X$ over $k$ we use $k(X)$ for the function field of $X$.
\resumetocwriting

\begin{acknowledgements}
	We would like to thank Tim Browning for many useful comments and for his suggestion on how to improve the upper bound in Proposition~\ref{prop:N_1}. We also thank Martin Bright, Dan Loughran, Yuri Manin and Bianca Viray for many useful discussions. We thank Pieter Moree for a helpful comment on the analytic part of the paper. We are grateful to the Max Planck Institute for Mathematics in Bonn and the Federal University of Rio de Janeiro for their hospitality while working on this article. Cec\'ilia Salgado was partially supported by FAPERJ grant E-26/203.205/2016,  the Serrapilheira Institute (grant Serra-1709-17759), Cnpq grant PQ2 310070/2017-1 and the Capes-Humboldt program.
\end{acknowledgements}

\section{Description of the Brauer group}
\label{sec:Brauer}
This section is dedicated to the study of $\Br X_\bfa / \Br \QQ$. We shall give a list of explicit representatives in $\Br X_\bfa$ of the elements generating $\Br X_\bfa / \Br \QQ$ in the case of cyclic group. This will later allow us to obtain the upper bounds in Theorems~\ref{thm:BMO} and \ref{thm:WA}. We follow a classical approach for computing the Brauer group of a conic bundle surface. We begin by embedding $X_\bfa$ in the scroll $\FF(1, 1, 0)$ following \cite[\S2]{Rei97}. A short summary of this is contained in \cite[\S2]{FLS18}. Our method closely follows \cite[\S3]{LM18}.

As explained in \cite[Ch.~2]{Bro09} if one of the quadrics defining $X_\bfa$ is of the shape $x_0x_1 - x_2x_3 = 0$, then there is a pair of morphisms $\pi_1: X_\bfa \rightarrow \PP^1$ and $\pi_2 : X_\bfa \rightarrow \PP^1$ defined over $\QQ$ each of which endows $X_\bfa$ with a different conic bundle structure. This can be seen in the following way. The map
\[
	\begin{split}
		\FF(1, 1, 0) &\rightarrow \PP^4, \\
		(s, t; x, y, z) &\mapsto (sx: ty: tx: sy: z)
	\end{split}
\]
defines an isomorphism between $X_\bfa$ and
\begin{equation}
	\label{eqn:conic bundle}
	(a_0s^2 + a_2t^2)x^2 + (a_3s^2 + a_1t^2)y^2 + a_4z^2 = 0 \subset \FF(1, 1, 0).
\end{equation}
One can view $\FF(1, 1, 0) = \PP(\sO_{\PP^1}(1)\oplus\sO_{\PP^1}(1)\oplus\sO_{\PP^1})$ as $((\AA^2 \setminus 0) \times (\AA^3 \setminus 0))/ \GG_m^2$, where the action of $\GG_m^2$ on $(\AA^2 \setminus 0) \times (\AA^3 \setminus 0)$ is described by
\[
	(\lambda, \mu) \cdot (s, t; x, y, z) = (\lambda s, \lambda t; \frac{\mu}{\lambda}x, \frac{\mu}{\lambda} y, \mu z).
\]
Then $\pi_1 : X_\bfa \rightarrow \PP^1$ is obtained by projecting to $(s, t)$. It is now clear that each fibre of $\pi_1$ is a conic and thus $X_\bfa$ is a conic bundle over the projective line. Similarly, one obtains $\pi_2 : X_\bfa \rightarrow \PP^1$ via the map $(s, t; x, y, z) \mapsto (tx: sy: ty: sx: z)$.

It follows from \eqref{eqn:conic bundle} that the conic associated to the generic fibre of $\pi_1$ takes the shape $-a_4(a_0s^2 + a_2t^2)x^2 - a_4(a_3s^2 + a_1t^2)y^2 - z^2 = 0$. There is an associated to it quaternion algebra $Q$ in the Brauer group of the function field of $\PP^1$ given by
\begin{equation}
	\label{def:gamma}
	Q = (-a_4(a_0(s/t)^2 + a_2), - a_4(a_3(s/t)^2 + a_1)).
\end{equation}
The quaternion algebra $Q$ has a trivial residue at any closed point of $\PP^1$ corresponding to a non-singular fibre of $\pi_1$. Its residues over the singular fibres of $\pi_1$ are described in the next lemma.

\begin{lemma} 
	\label{lem:conic_bundle}
	The following holds.
	\begin{enumerate}[label=\emph{(\roman*)}]
		\item The map $\pi_1: X_\bfa \to \PP^1$ has $4$ singular geometric fibres. 
		\item The bad fibres lie over the zero locus of 
		\[
			\Delta(s,t) = (a_0s^2 + a_2t^2)(a_3s^2 + a_1t^2).
		\]
		\item Assume that $-a_0a_2, -a_1a_3 \notin \QQ^{*2}$. Let $T'$, $T''$ be the closed points corresponding the zero locus of $a_0s^2 + a_2t^2$ and $a_3s^2 + a_1t^2$, respectively. They have residue fields $\QQ(T') = \QQ(\sqrt{-a_0a_2})$ and $\QQ(T'') = \QQ(\sqrt{-a_1a_3})$. The fibres over $T', T''$ have the following residues:
		\[
			\begin{split}
				\Res_{T'}(Q) &= -a_0a_4d \in \QQ(T')^*/\QQ(T')^{*2}, \\
				\Res_{T''}(Q) &= -a_1a_4d \in \QQ(T'')^*/\QQ(T'')^{*2}.
			\end{split}
		\]
	\end{enumerate}
\end{lemma}

\begin{proof}
	The result follows immediately from the explicit equation \eqref{eqn:conic bundle} and a simple calculation.
\end{proof}

We continue with the structure of $\Br X_\bfa / \Br \QQ$ given in the next proposition.

\begin{proposition}
	\label{prop:BrXconic}
	Let $(*)$ denote the condition that  $-a_0a_4d \notin \QQ(\sqrt{-a_0a_2})^{*2}$, $-a_1a_4d \notin \QQ(\sqrt{-a_1a_3})^{*2}$ and that one of $-a_0a_2$, $-a_1a_3$ or $a_0a_1$ is not in $\QQ^{*2}$.	Then we have
	\[
		\Br X_{\bfa} / \Br \QQ =
		\begin{cases}
			(\ZZ/2\ZZ)^2 &\mbox{if } a_0a_1, a_2a_3, -a_0a_2 \in \QQ^{*2} \mbox{ and } -a_0a_4d \not\in \QQ^{*2}, \\
			\ZZ/2\ZZ &\mbox{if } (*),\\ 
			\{\id\} &\mbox{otherwise.}
		\end{cases}
	\]
\end{proposition} 

It is worth mentioning that under the assumption $X_\bfa(\bfA_{\QQ}) \neq \emptyset$ we have $\Br X_\bfa / \Br \QQ$ trivial if and only if $X_\bfa$ is rational. This is shown in our forthcoming work. We now proceed with the proof of Proposition~\ref{prop:BrXconic}.

\begin{proof}
	The morphism $\pi_1: X_\bfa \to \PP^1$ defines a pull-back $\pi_1^*$ on the level of Brauer groups. Thus Faddeev's reciprocity law and Grothendieck's purity theorem give the following commutative diagram
	\begin{equation}
		\label{eqn:Faddeev+Purity}
		\begin{tikzcd}
			0 \ar{r}
			&\Br \QQ \ar{r} \ar{d}{\pi_1^*}
			&\Br \QQ(T) \ar{r} \ar{d}{\pi_1^*}
			& \bigoplus \HH^1(\QQ[T]/(P(T)), \QQ/\ZZ) \\
			0 \ar{r} 
			&\Br X_\bfa \ar{r}
			&\Br \QQ(X_\bfa) \ar{r}
			& \bigoplus \HH^1(\QQ(Y), \QQ/\ZZ),
		\end{tikzcd}
	\end{equation}
	where the top sum is taken over all irreducible polynomials $P(T) \in \QQ[T]$ and the bottom sum is taken over all integral subvarieties $Y$ of $X_\bfa$ of codimension 1. 

	If $\alpha \in \Br X_\bfa$, then the image of $\alpha$ in the Brauer group $\Br \QQ(X_\bfa)$ of the function field of $X_\bfa$ is the pull-back $\pi_1^*A$ of some $A \in \Br \QQ(T)$ by \cite[Thm.~2.2.1]{CTSW90}. Proposition~\ref{prop:BrXconic} now follows from Lemma~\ref{lem:conic_bundle} and a standard computation of residues.	
\end{proof}

\begin{remark}
	\label{rem:Br4}
	Note that in the case when $-a_0a_2 \in \QQ^{*2}$ there is an obvious rational point $(s,t;x,y,z) = (1:\sqrt{-a_0/a_2};1:0:0)$ on $X_\bfa$. Thus there is no Brauer--Manin obstruction to the existence of rational points on $X_\bfa$ if $\Br X_\bfa / \Br \QQ$ is of order four. 
\end{remark}

\subsection{Generators}
In order to study the Brauer--Manin obstruction to the Hasse principle or to weak approximation we can assume that $X_\bfa$ has a non-trivial Brauer group, that is $-a_1a_4d \notin \QQ(\sqrt{-a_1a_3})^{*2}, -a_0a_4d \notin \QQ(\sqrt{-a_0a_2})^{*2}$. By Theorem~\ref{thm:Br} the number of surfaces with Brauer group of order four is $O(B^{3})$ which negligible compared to the bounds needed to establish Theorems~\ref{thm:BMO} and \ref{thm:WA}. Thus we can assume from now on that $\Br X_\bfa / \Br \QQ \simeq \ZZ/2\ZZ$. 

Let $A \in \Br \QQ(\PP^1)$ be given by
\[
 	A
 	=	(a_0(s/t)^2 + a_2, -a_0a_4d).
\]
It is clear that the image $\alpha = \pi_1^*A \in \Br \QQ(X_\bfa)$ of $A$ is unramified along each irreducible divisor of $X_\bfa$ except possibly on the irreducible components of $D = \{a_0(s/t)^2 + a_2 = 0\} \subset X_\bfa$. To fix ideas assume that $-a_0a_2$ is not a rational square, hence $D$ is irreducible over $\QQ$. A similar analysis yields the same conclusion when $-a_0a_2$ is a square. Along $D$ one has
\[
	-a_0a_4d = (a_0a_4z/ty)^2.
\]
Thus $\Res_{\QQ(D)}\alpha = -a_0a_4d \in \QQ(D)^*/\QQ(D)^{*2}$ is trivial which implies that $\alpha$ is unramified on $D$. Alternatively, one can check that $\QQ(D) = \QQ(\sqrt{-a_0a_2}, \sqrt{-a_0a_4d})(T)$ where $-a_0a_4d$ is clearly a square. Since $X_\bfa$ is smooth we can apply Grothendieck's purity theorem (the bottom line of \eqref{eqn:Faddeev+Purity}) to conclude that $\alpha$ lies inside the image of $\Br X_\bfa \rightarrow \Br \QQ(X_\bfa)$. A similar analysis as in Lemma~\ref{lem:conic_bundle}~(iii) shows that the image of $\alpha$ in $\Br X_\bfa / \Br \QQ$ is non-trivial and hence it generates $\Br X_\bfa / \Br \QQ \simeq \ZZ/2\ZZ$.

\section{Local points}
\label{sec:local}
\subsection{Local solubility}
\label{subsec:local points}
In this subsection we are concerned with the existence of local points on $X_\bfa$ given in \eqref{eq:dP4 main}. To do so we first define an equivalence relation on $\ZZ_{\ge 0}^5$ in the spirit of \cite[\S2]{BBL16}. We say that 
\[
	(\alpha_0, \alpha_1, \alpha_2, \alpha_3, \alpha_4) \sim (\beta_0, \beta_1, \beta_2, \beta_3, \beta_4) 
\]
if and only if at least one of the following holds.
\begin{itemize}
	\item
	$(\alpha_0, \alpha_1, \alpha_2, \alpha_3, \alpha_4) = (\beta_1, \beta_0, \beta_2, \beta_3, \beta_4)$,

	\item
	$(\alpha_0, \alpha_1, \alpha_2, \alpha_3, \alpha_4) = (\beta_2, \beta_3, \beta_0, \beta_1, \beta_4) $,
	
	\item
	$(\alpha_0, \alpha_1, \alpha_2, \alpha_3) = (\beta_0, \beta_1, \beta_2, \beta_3,)$ and $\alpha_4 \equiv \beta_4 \bmod 2$.

	\item 
	There are $k, \ell, m, n \in \ZZ$ satisfying $k + \ell = m + n$ such that
	\[
		(\alpha_0, \alpha_1, \alpha_2, \alpha_3, \alpha_4) = (\beta_0 + 2k, \beta_1 + 2\ell, \beta_2 + 2m, \beta_3 + 2n, \beta_4).
	\]

	\item 
	There is some $k \in \ZZ$ such that 
	\[
		(\alpha_0, \alpha_1, \alpha_2, \alpha_3, \alpha_4) = (\beta_0 + k, \beta_1 + k, \beta_2 + k, \beta_3 + k, \beta_4 + k).
	\]
\end{itemize}

Let $p$ be a prime number and let $\valp$ denote the $p$-adic valuation of an element of $\Qp$. The above equivalence relation has the property that for each $\bfa, \bfa' \in \ZZ_{\ge 0}^5$ with
\[
	(\valp(a_0), \valp(a_1), \valp(a_2), \valp(a_3), \valp(a_4)) \sim (\valp(a_0'), \valp(a_1'), \valp(a_2'), \valp(a_3'), \valp(a_4')) 
\]
we have $X_\bfa(\Qp) \neq \emptyset$ if and only if $X_{\bfa'}(\Qp) \neq \emptyset$. We will make great use of this fact. Unlike in \cite[\S2]{BBL16} in our setting when we quotient $\ZZ_{\ge 0}^5$ by the equivalence relation $\sim$ we do not get a finite list of representatives. Thus we require a more involved approach in order to understand the local solubility of $X_\bfa$. For convenience, if $p$ is an odd prime let
\[
 \left[\frac{a}{p}\right] = \(\frac{a/p^{\valp(a)}}{p}\),
\]
where the second entry is the Legendre symbol. We shall give necessary and sufficient conditions for the existence of local points on $X_\bfa$ in the next proposition.

\begin{proposition}
	\label{lem:local solubility}
	Let $p \neq 2$ be a place of $\QQ$. Then we have $X_\bfa(\Qp) = \emptyset$ if and only if one of the following holds.
	\begin{enumerate}[label=\emph{(\roman*)}]
		\item $p = \infty$ and all $a_i$ have the same sign,
		
		\item $p = 3$, and the following hold
		\begin{itemize}
			\item $\val3(a_0) \equiv \val3(a_1) \equiv \val3(a_2) \equiv \val3(a_3) \nequiv \val3(a_4) \bmod 2$,

			\item $\val3(a_0a_1) = \val3(a_2a_3)$,

			\item $\left[\frac{-a_0 a_2}{3}\right] = \left[\frac{-a_0 a_3}{3}\right] = \left[\frac{-a_1 a_2}{3}\right] = \left[\frac{-a_1 a_3}{3}\right] = -1$.
		\end{itemize}

		\item $p$ is an odd prime, $(i, j)$ is either $(0, 1)$ or $(2, 3)$, we have $\{\ell, m\} = \{0, 1, 2, 3\} \setminus \{i, j\}$ and the following hold
		\begin{itemize}
			\item $\valp(a_i) \equiv \valp(a_j) \equiv \valp(a_4) \bmod 2$, 
			
			\item $\valp(a_\ell) \equiv \valp(a_m) \nequiv \valp(a_i) \bmod 2$, 
			
			\item $\valp(a_ia_j) = \valp(a_\ell a_m)$ and $\left[\frac{-a_i a_4}{p}\right] = \left[\frac{-a_j a_4}{p}\right] = -1$ or \\
			$\valp(a_ia_j) > \valp(a_\ell a_m)$ and $\left[\frac{-a_i a_4}{p}\right] = \left[\frac{-a_j a_4}{p}\right] = \left[\frac{-a_\ell a_m}{p}\right] = -1$.
		\end{itemize}

		\item $p$ is an odd prime, $(i, j, k)$ is one of $(0, 1, 2)$, $(0, 1, 3)$, $(2, 3, 0)$ or $(2, 3, 1)$ as an ordered triple, $\{\ell\} = \{0, 1, 2, 3\} \setminus \{i, j, k\}$ and the following hold
		\begin{itemize}
			\item $\valp(a_i) \equiv \valp(a_j) \equiv \valp(a_k) \bmod 2$, 
			
			\item $\valp(a_\ell) \equiv \valp(a_4) \nequiv \valp(a_i) \bmod 2$, 
	
			\item $\valp(a_i a_j) > \valp(a_k a_\ell)$,
	
			\item $\left[\frac{-a_i a_k}{p}\right] = \left[\frac{-a_j a_k}{p}\right] = \left[\frac{-a_\ell a_4}{p}\right] = -1$.
		\end{itemize}
	\end{enumerate}
\end{proposition}

\begin{proof}
	One easily verifies that if $p = \infty$, then $X(\RR) \neq \emptyset$ if and only if two of the $a_i$ have different signs since this is equivalent to the bottom quadratic from in \eqref{eq:dP4 main} being indefinite.

	Let $p$ be an odd prime now. It is clear that at least three of the $a_i$ have the same parity of their $p$-adic valuations. Let $\alpha = (\valp(a_0), \valp(a_1), \valp(a_2), \valp(a_3), \valp(a_4))$. We distinguish between the following cases.
	\begin{enumerate}[label=(\alph*)]
		\item 
		If $\valp(a_0) \equiv \valp(a_2) \equiv \valp(a_4) \bmod 2$ we set $x_1 = x_3 = 0$. It thus suffices to show that the diagonal projective conic $a_0x_0^2 + a_2x_2^2 + a_4x_4^2 = 0$ has a $p$-adic point. In view of $\sim$ since $\valp(a_0), \valp(a_2), \valp(a_4)$ have the same parity we can assume that $p \nmid a_0a_2a_4$. Such conics are known to have a smooth $\Fp$-point which is easily verified for example by fixing $x_4 \in \Fp^*$ and then counting the number of possible values that $a_0x_0^2$ and $a_2x_2^2 + a_4x_4^2$ can take. Moreover, a smooth $\Fp$-point on the conic lifts to a $\Zp$-point on $X_\bfa$ with $x_1 = x_3 = 0$ by Hensel's lemma. Indeed, the Jacobian matrix of $X_\bfa$ is
		\[
			J(\bfx) = \left(
			\begin{array}{ccccc}
				x_1 & x_0 & -x_3 & -x_2 & 0 \\
				2a_0x_0 & 2a_1x_1 & 2a_2x_2 & 2a_3x_3 & 2a_4x_4
			\end{array}
			\right).
		\]
		The above argument shows that either the minor $J_{1,2}$ or the minor $J_{3,4}$ has a determinant a unit in $\Fp$ when evaluated at $\bfx$. The same analysis applies if $a_0, a_3, a_4$ or $a_1, a_2, a_4$ or $a_1, a_3, a_4$ have the same parity of their $p$-adic valuations.

		\item 
		If $\valp(a_0) \equiv \valp(a_1) \equiv \valp(a_4) \bmod 2$ we can assume that $\valp(a_2) \equiv \valp(a_3) \nequiv \valp(a_0) \bmod 2$, otherwise we fall into (a). If $p \nmid (a_0, a_1, a_4)$ we can apply a similar argument to the one in (a) by solving $a_0x_0^2 + a_1x_1^2 + a_4x_4^2 \equiv 0 \bmod p$ first and then taking $x_2, x_3$ such that $x_2x_3 \equiv x_0x_1 \bmod p$. Thus $X_\bfa(\Qp) \neq \emptyset$. Moreover, if $\valp(a_0a_1) < \valp(a_2a_3)$, then $\alpha \sim (0, 0, 2k + 1, 1, 0)$ for some $k \ge 0$ and hence $X_\bfa(\Qp) \neq \emptyset$. 
		
		Assume that $\valp(a_0a_1) \ge \valp(a_2a_3)$. Recall \eqref{eqn:conic bundle}. Our investigation continues with an analysis of the following system
		\[
			\begin{split}
				a_0(sx)^2 + a_1(ty)^2 + a_4z^2 = 0,\\
				a_2(tx)^2 + a_3(sy)^2 = 0.
			\end{split}
		\]
		We have $X_\bfa(\Qp) \neq \emptyset$ if $\left[\frac{-a_0a_4}{p}\right] = 1$ by setting $t = y = 0$. Similarly, $X_\bfa(\Qp) \neq \emptyset$ if $\left[\frac{-a_1a_4}{p}\right] = 1$. 

		If $\valp(a_0a_1) = \valp(a_2a_3)$, then $\alpha \sim (2, 0, 1, 1, 0)$ and hence one can conclude that $X_\bfa(\Qp) = \emptyset$ if $\left[\frac{-a_0a_4}{p}\right] = \left[\frac{-a_1a_4}{p}\right] = -1$ by looking at the $p$-adic valuation of possible solutions to the system above. 

		On the other hand, if $\valp(a_0a_1) > \valp(a_2a_3)$, then $\alpha \sim (2m, 2, 1, 1, 2)$ for some $m > 0$ and hence $X_\bfa(\Qp) \neq \emptyset$ if $\left[\frac{-a_2a_3}{p}\right] = 1$ since this condition implies the existence of a smooth $\Fp$-point on $X_\bfa$ mod $p$.

		We assume from now on that $\valp(a_0a_1) > \valp(a_2a_3)$ and
		\[
			\left[\frac{-a_0a_4}{p}\right] 
			= \left[\frac{-a_1a_4}{p}\right] 
			= \left[\frac{-a_2a_3}{p}\right]
			= -1.
		\]
		We claim that $X_\bfa(\Qp) = \emptyset$. If there was a point in $X_\bfa(\Qp)$, then by looking at the possible $p$-adic valuations of its coordinates we see that at least one of $\left[\frac{-a_0a_4}{p}\right]$, $\left[\frac{-a_1a_4}{p}\right]$ or $\left[\frac{-a_2a_3}{p}\right]$ has to be trivial, a contradiction! We treat the case $\valp(a_2) \equiv \valp(a_3) \equiv \valp(a_4) \bmod 2$ in a similar way.

		\item
		If $\valp(a_0) \equiv \valp(a_1) \equiv \valp(a_2) \nequiv \valp(a_3) \equiv \valp(a_4) \bmod 2$, then the analysis is very similar to the one in (b) modulo the fact that there is no symmetry in $a_3$ and $a_4$. Here \eqref{eqn:conic bundle} transforms into
		\[
			\begin{split}
				a_0(sx)^2 + a_1(ty)^2 + a_2(tx)^2 = 0,\\
				a_3(sy)^2 + a_4z^2 = 0.
			\end{split}
		\]
		Once again we begin by observing that $X_\bfa(\Qp) \neq \emptyset$ unless
		\[
			\left[\frac{-a_0a_2}{p}\right] = \left[\frac{-a_1a_2}{p}\right] = \left[\frac{-a_3a_4}{p}\right] = -1,
		\]
		which we assume from now on. It is clear that $\alpha \sim (2m, 0, 0, 2\ell + 1, 1)$ for some $m, \ell \ge 0$. If $2m < 2\ell + 1$, then we reduce once more to finding a $p$-adic point on the conic $a_0u^2 + a_1v^2 + a_2w^2 = 0$ and thus $X_\bfa(\Qp) \neq \emptyset$. On the other hand, if $2m > 2\ell + 1$, then as in (b) we see that $X_\bfa(\Qp) = \emptyset$. An identical analysis applies in the remaining cases of (v).

		\item
		If $\valp(a_0) \equiv \valp(a_1) \equiv \valp(a_2) \equiv \valp(a_3) \nequiv \valp(a_4) \bmod 2$, then via the above relation we can assume that $\alpha = (2m , 0, 0, 0, 1)$ for some $m > 0$ if $\valp(a_0a_1) > \valp(a_2a_3)$. Once again by \eqref{eqn:conic bundle} we need to analyse
		\begin{equation}
			\label{eqn:case d}
			a_0(sx)^2 + a_1(ty)^2 + a_2(tx)^2 + a_3(sy)^2 = 0.
		\end{equation}
		It thus suffices to show that $a_1u^2 + a_2v^2 + a_3w^2 = 0$ has a $p$-adic point. As explained in (a) this is always the case. A similar argument applies when $\valp(a_0a_1) < \valp(a_2a_3)$.

		Assume now that $\valp(a_0a_1) = \valp(a_2a_3)$ and thus we have $\alpha \sim (0, 0, 0, 0, 1)$. In view of \eqref{eqn:case d} there is a smooth $\Fp$-point on $X_\bfa$ if $-a_ia_j \bmod p \in \Fp^{*2}$ for some $(i, j) \in \{(0, 2), (0, 3), (1, 2), (1, 3)\}$. On the other hand, if the reduction of all of the above $-a_ia_j \bmod p$ is a non-square in $\Fp$, then any $\Fp$-point must satisfy $tx \neq 0$. The change of variables $s/t = X$, $y/x = Y$ reduces the problem to showing the existence of an $\Fp$-point on $Y^2 = -(a_0X^2 + a_2)/(a_3X^2 + a_1)$. In fact, our assumptions imply that any $\Fp$-point on it must be smooth if it exists. Clearly, the existence of such an $\Fp$-point is equivalent to the existence of an $\Fp$-point on the genus one curve
		\[
			C: \quad Y^2 = -(a_0X^2 + a_2)(a_3X^2 + a_1).
		\]
		The Hasse--Weil bound implies that $\#C(\Fp) \ge p + 1 - 2\sqrt{p}$. This quantity is clearly positive if $p \ge 5$. On the other hand, if $p = 3$, then the condition $-a_ia_j \bmod p$ is a non-square in $\Fp$ for each $(i, j) \in \{(0, 2), (0, 3), (1, 2), (1, 3)\}$ says that $a_0 \equiv a_1 \equiv a_2 \equiv a_3 \bmod 3$. One easily checks that in this case $X_\bfa(\QQ_3) = \emptyset$. This completes the proof.
	\end{enumerate}
\end{proof}

\subsection{Local densities}
We continue with the study of the proportion of everywhere locally soluble $X_\bfa \in \sF$. For each place $p$ of $\QQ$ let $X_{\bfa, p} = X_\bfa \times_\QQ \Qp$ and define
\[
	\begin{split}
		\Omega_p
		&= \left\{ \bfa \in \Zp^5 \ : \ p \nmid \bfa, X_{\bfa, p} \mbox{ smooth and } X_{\bfa, p}(\Qp) \neq \emptyset \right\}, \\
		\Omega_\infty
		&= \left\{ \bfa \in [-1, 1]^5  \ : \ X_{\bfa, \infty} \mbox{ smooth and } X_{\bfa, \infty}(\RR) \neq \emptyset \right\}.
	\end{split}
\]
Let $\mu_p$ be the normalised Haar measure on $\Zp^5$ such that $\mu_p\( \Zp^5 \) = 1$ and let $\mu_{\infty}$ be the Lebesgue measure on $\RR^5$. The local densities $\sigma_p$ corresponding to $X_\bfa$ are then defined by
\[
	\sigma_p
	=\mu_p\(\Omega_p\), \quad
	\sigma_{\infty}
	= \frac{\mu_\infty(\Omega_\infty)}{\mu_{\infty}(\left\{ \bfa \in [-1, 1]^5  \ : \  X_{\bfa, \infty} \mbox{ smooth}\right\})}.
\]
We note that this definition of local densities agrees with the one given in \cite{BBL16}. 

We shall estimate $\sigma_p$ in the next proposition. For technical reasons we avoid the calculation of $\sigma_2$ here. However, it is easy to see that $\sigma_2 > 0$, for example by calculating the proportion of $\bfa \in \ZZ_2^5$ such that $a_i$ are all units in $\ZZ_2$ and $a_1 \equiv -a_4 \bmod 8$.

\begin{proposition}
	\label{prop:local densities}
	We have $\sigma_\infty = 15/16$ and for each odd prime $p$ the following holds.
	\[
		\sigma_p = 
		\begin{cases}
		\frac{63693071}{66355200} &\mbox{if } p = 3, \\
		1 - \frac{1}{2 p^2} + \frac{9}{4 p^3} + O\(\frac{1}{p^4}\) &\mbox{if } p > 3.
		\end{cases}
	\]
\end{proposition}

\begin{proof}
	One clearly has
	\[
		\frac{\mu_\infty(\{\bfa \in [-1, 1]^5 \ : \ (a_0a_1 - a_2a_3) \prod_{i = 0}^4a_i = 0\})}{\mu_\infty([-1, 1]^5)} = 0,
	\]
	since each of the conditions $a_i = 0$ or $a_0a_1 - a_2a_3 = 0$ defines a proper subspace of $\RR^5$. Thus
	\[
		\sigma_\infty
		= \frac{\mu_\infty(\Omega_\infty)}{\mu_{\infty}([-1, 1]^5)}.
	\]
	On the other hand, by Lemma~\ref{lem:local solubility}(i) we have $X_{\bfa, \infty}(\RR) = \emptyset$ if and only if all $a_i$ have the same sign. Thus 
	\[
		\sigma_\infty 
		= 1 - \frac{\mu_\infty (\{\bfa \in [-1, 1]^5 \ : \ \mbox{all $a_i$ have the same sign} \})}{\mu_\infty([-1,1]^5)}
		= 1 - \frac{1}{2^4} = \frac{15}{16}, 
	\]
	as claimed. 

	Assume now that $p$ is an odd prime. Once again we have
	\[
		\mu_p(\{\bfa \in \Zp^5 \ : \ (a_0a_1 - a_2a_3) \prod_{i = 0}^4a_i = 0\}) = 0,
	\]
	which allows us to safely ignore this condition from now on. The conditions on local solubility for $X_{\bfa, p}$ in this case are given in Lemma~\ref{lem:local solubility}(iv), (v) and in (iii) if $p = 3$. We continue with calculating the proportion of surfaces $X_{\bfa, p}$ with $\bfa \in \Zp^5$ satisfying each of these conditions. 

	We shall explain how to compute the contribution coming from Lemma~\ref{lem:local solubility}~(iii), the analysis of the other cases being similar. Firstly, we partition the set $S^{(iii)} = S_1^{(iii)} \cup S_1^{(iii)}$ of all $\bfa \in \Zp$ satisfying Lemma~\ref{lem:local solubility}(iii) into two disjoint subsets. Here $S_1^{(iii)}$ is the subset of $S^{(iii)}$ consisting of those $\bfa$ with $\valp(a_4)$ odd and $S_2^{(iii)}$ is its complement in $S^{(iii)}$. To compute the measure of $S_1^{(iii)}$ let
	\[
		\begin{split}
		\valp(a_0) &= 2k, \quad
		\valp(a_1) = 2\ell, \quad 
		\valp(a_2) = 2m, \\
		\valp(a_3) &= 2n, \quad
		\valp(a_4) = 2r + 1.
		\end{split}
	\]
	It is clear that under the above parametrisation the condition $p \nmid \bfa$ is equivalent to the minimum of $k, \ell, m, n$ being $0$.	We have $\left[\frac{-a_0 a_2}{3}\right] = \left[\frac{-a_0 a_3}{3}\right] = \left[\frac{-a_1 a_2}{3}\right] = \left[\frac{-a_1 a_3}{3}\right] = -1$ with probability $1/8$. The proportion of $a \in \ZZ_p$ with $\valp(a) = t$ is $p^{-t}(1 - 1/p)$. Thus we have
	\[
		\begin{split}
		\mu_p(S_1^{(iii)}) 
		&= \frac{1}{8}\(1 - \frac{1}{p}\)^5
		\sum_{r \ge 0} \frac{1}{p^{2r + 1}}
		\(\sum_{\substack{k, \ell, m, n \ge 0 \\ k + \ell = m + n}} \frac{1}{p^{2k + 2\ell + 2m + 2n}}
		- \sum_{\substack{k, \ell, m, n > 0 \\ k + \ell = m + n}} \frac{1}{p^{2k + 2\ell + 2m + 2n}}\) \\
		&= \frac{(p - 1)^2 (p^4 + 1)^2}{8 p^4 (p + 1)^3 (p^2 + 1)^2}.
		\end{split}
	\]

Applying the same circle of ideas to $S_2^{(iii)}$ one gets
		\[
		\begin{split}
		\mu_p(S_2^{(iii)}) 
		= \frac{(p - 1)^2 (p^4 + 1)}{8 p (p + 1)^3 (p^2 + 1)^3}.
		\end{split}
	\]
	Thus
	\[
		\mu_p(S^{(iii)})
		= \mu_p(S_1^{(iii)}) + \mu_p(S_2^{(iii)})
		= \frac{(p - 1)^2 (p^2 - p + 1) (p^4 + 1) (p^4 + p^3 + p^2 + p + 1)}{8 p^4 (p + 1)^3 (p^2 + 1)^3}.
	\]

A similar analysis shows that
	\begin{equation*}
	\begin{split}
		\mu_p(S^{(iv)})
		&= \frac{(p + 1) (4 p^4 + p^2 - 2) (p - 1)^6 + p^4 (p^8 + 6 p^6 + 4 p^4 + 1)}{4 p (p^2 - 1)^4 (p^2 + 1)^3}, \\
		\mu_p(S^{(v)}) 
		&= \frac{(p - 1) (p^8 + 4 p^6 + 7 p^4 + 5 p^2 + 3)}{2 p (p + 1)^4 (p^2 + 1)^3}.
	\end{split}
	\end{equation*}

	What is left is to take into account that we have the following expression
	\[
		\sigma_p
		=
		\begin{cases}
			1 - \mu_p(S^{(iii)}) - \mu_p(S^{(iv)}) - \mu_p(S^{(v)}), &\mbox{if } p = 3, \\
			1 - \mu_p(S^{(iv)}) - \mu_p(S^{(v)}, &\mbox{if } p > 3.
		\end{cases}
	\]
	Thus $\sigma_3 = 63693071/66355200$ and if $p > 3$ we get
	\[
		\begin{split}
			\sigma_p 
			= 1 - \frac{1}{2 p^2} + \frac{9}{4 p^3} + O\(\frac{1}{p^4}\).
		\end{split}
	\]
	This completes the proof of Proposition~\ref{prop:local densities}.
\end{proof}

\section{Proof of Theorem~\ref{thm:local solubility}} 
\label{sec:thm local}
We will now prove Theorem~\ref{thm:local solubility}. It suffices to apply \cite[Thm.~1.3]{BBL16} to the morphism $f : \sF \rightarrow \PP^4$ which projects each $X_\bfa$ to its coordinate vector $\bfa$. In order to do so we need that $\sF(\bfA_\QQ) \neq \emptyset$ and that the fibre of $f$ above each codimension one point of $\PP^4$ is split, i.e., it contains a geometrically integral open subscheme. The former is assured by our arrangements, while for the latter it is enough to check that the singular fibres of $f$ are split. The singular locus is determined by $a_0\cdots a_4(a_0a_1-a_2a_3)=0$.  We do a case by case analysis and deal with the following separately.
\begin{enumerate}[label=(\roman*)]
	\item $a_0=0$, 
	\item $a_4=0$,
	\item $a_0a_1-a_2a_3=0$.
\end{enumerate}
Notice that $a_i=0$, for $i=1,2,3$, is analogous to case (i) thanks to the symmetry in the equations defining $X_\bfa$.

For case (i) the fibre is defined via
\[
	\begin{split}
		x_0x_1-x_2x_3=0, \\
		a_1x_1^2 + \cdots + a_4x_4^2=0.
	\end{split}
\]
Consider the chart $x_1=1$. In it, $x_0$ is determined by $x_0=x_2x_3$ and thus the fibre is birational to a smooth quadric surface in $\mathbb{P}^3$ which is clearly split.
For cases (ii) and (iii), we use the conic bundle representation. The fibres have equations given by:
\[
	\begin{split}
		&\text{(ii):}\quad (a_0s^2+a_2t^2)x^2+(a_3s^2+a_1t^2)y^2=0, \\
		&\text{(iii):}\quad (a_0s^2+a_2t^2)(x^2+y^2)+ a_4z^2=0.
	\end{split}
\]
Notice that both are irreducible. Indeed, otherwise the generic fibres of each conic bundle would be reducible and hence singular. This is clearly not the case as both admit smooth fibres. Thus all conditions of \cite[Thm.~1.3]{BBL16} are fulfilled and Theorem~\ref{thm:local solubility} follows from it. We thank Dan Loughran for pointing out the above argument to us. 
\qed

\section{Surfaces with Brauer group of order 4}
\label{sec:br4}
This section is dedicated to the proof of the asymptotic formula for $N_4(B)$ appearing in Theorem~\ref{thm:Br}. Recall that $d = a_0a_1 - a_2a_3$. We first need to understand how often $d = 0$ under the assumptions $a_0a_1$, $-a_0a_2$, $a_2a_3 \in \QQ^{*2}$. We claim that
\begin{equation}
	\label{eqn:d=0}
	\#\left\{(a_0, \dots, a_3) \in \ZZ^4 \cap [-B, B]^4 \ : \ a_0a_1, -a_0a_2 \in \QQ^{*2} \mbox{ and } d = 0 
	\right\} 
	\ll B (\log B)^3.
\end{equation}
Indeed, for a quadruple $(a_0, \dots, a_3) \in \ZZ^4$ the conditions $a_0a_1, -a_0a_2, a_2a_3 \in \QQ^{*2}$ imply that $a_0, \dots, a_3$ must obey the following factorisation
\[
	a_0 = mk^2b_0^2, \quad
	a_1 = mk^2b_1^2, \quad
	a_2 = -m\ell^2b_2^2, \quad
	a_3 = -m\ell^2b_3^2,
\]
with $(k, \ell) = (b_0, b_1) = (b_2, b_3) = 1$. It suffices to consider the case where $k, \ell, m, b_0, b_1, b_2, b_3$ are all positive integers. 

The condition $a_0a_1 = a_2a_3$ becomes now $k^2b_0b_1 = \ell^2b_2b_3$ and since $(k, \ell) = 1$ we must have $k^2 \mid b_2b_3$ and $\ell^2 \mid b_0b_1$. Write $k = k_2k_3$ and $\ell = \ell_0\ell_1$ so that $b_0 = \ell_0^2c_0$, $b_1 = \ell_1^2c_1$, $b_2 = k_2^2c_2$ and $b_3 = k_3^2c_3$. Thus $c_0c_1 = c_2c_3$. Writing $c_0 = r_0s_0$ and $c_1 = r_1s_1$ now gives
\[
	a_0 = m k_2^2 k_3^2 \ell_0^4 r_0^2 s_0^2, \quad
	a_1 = m k_2^2 k_3^2 \ell_1^4 r_1^2 s_1^2, \quad
	a_2 = m \ell_0^2 \ell_1^2 k_2^4 r_0^2 r_1^2, \quad
	a_3 = m \ell_0^2 \ell_1^2 k_3^4 s_0^2 s_1^2,
\]
with $(k_2k_3, \ell_0\ell_1) = (\ell_0r_0d_0, \ell_1r_1d_1) = (k_2r_0r_1, k_3d_0d_1) = 1$. Forgetting the coprimality conditions and summing over $r_0, r_1$ first proves the claim. We continue now with the study of $N_4(B)$.

\begin{proposition}
	\label{prop:Br4}
	We have
	\[
		N_4(B) 
		= \frac{60}{\pi^2} B^3 
		+ O(B^{5/2}(\log B)^2).
	\]
\end{proposition}

\begin{proof}
	Proposition~\ref{prop:BrXconic} implies that the quotient $\Br X_{\bfa} / \Br \QQ$ is of order $4$ if and only if $a_0a_1$, $-a_0a_2$, $a_2a_3$ are all rational squares and $-a_0a_4d \notin \QQ^{*2}$. Moreover, if these conditions are met, then $(1:0:\sqrt{-a_0/a_2}:0:0) \in X_\bfa(\QQ)$ and thus we need not worry about local solubility. One easily verifies that $a_0, a_1$ need to have the same sign, and similarly for $a_2, a_3$ while $a_0$, $a_2$ need to have different signs. Thus we can distinguish between four cases depending on the sign of each $a_i$, these are
	\[
		\begin{array}{cc}
			\mbox{(i)} \ a_0, a_1, a_4 > 0, \ a_2, a_3 < 0, 
			&\mbox{(iii)} \ a_0, a_1 > 0, \ a_2, a_3, a_4 < 0, \\
			\mbox{(ii)} \  a_0, a_1, a_4 < 0, \ a_2, a_3 > 0,
			&\mbox{(iv)} \ a_0, a_1 < 0, a_2, \ a_3, a_4 > 0.
		\end{array}
	\]
	
	One checks that (i) and (ii) contribute to the same amount in $N_4(B)$ via the map $\bfa \mapsto -\bfa$ and so do (iii) and (iv). On the other hand, (i) and (iv) have equal contribution in $N_4(B)$ as seen by the map $(a_0, a_1, a_2, a_3, a_4) \mapsto (a_2, a_3, a_0, a_1, a_4)$. Thus if $N_4^{(i)}(B)$ is the contribution in $N_4(B)$ coming from (i), we then clearly have
	\begin{equation}
		\label{eq:N_4 N_4(i)}
		N_4(B) 
		= 4N_4^{(i)}(B).
	\end{equation}

	We continue with the study of $N_4^{(i)}(B)$. By definition it equals the following quantity
	\[
		\# \{ \bfa \in \ZZ^5 : 0 < a_0, a_1, -a_2, - a_3, a_4 \le B, a_0a_1, -a_0a_2, a_2a_3 \in \QQ^{*2} \mbox{ and } -a_0a_4d \notin \QQ^{*2}\}.
	\] 
	Let us first make the following change $(a_2, a_3) \mapsto (-a_2, -a_3)$ so that all of the coordinates of $\bfa$ are positive integers. Thus
	\[
		\begin{split}
			N_4^{(i)}(B)
			&= \sum_{\substack{a_0, \dots, a_4 \le B \\ a_0a_1, a_0a_2, a_2a_3 \in \QQ^{*2} \text{ and } -a_0a_4d \notin \QQ^{*2}}} 1 \\
			&= \sum_{\substack{a_0, \dots, a_4 \le B \\ a_0a_1, a_0a_2, a_2a_3 \in \QQ^{*2}}} 1
			- \sum_{\substack{a_0, \dots, a_4 \le B \\ -a_0a_4d, a_0a_1, a_0a_2, a_2a_3 \in \QQ^{*2}}} 1
			- \sum_{\substack{a_0, \dots, a_4 \le B \\ d = 0, a_0a_1, a_0a_2, a_2a_3 \in \QQ^{*2}}} 1.
		\end{split}
	\] 
	By \eqref{eqn:d=0} the last sum is $O(B^2(\log B)^3)$, the extra exponent of $B$ coming from the additional sum over $a_4$. We shall soon see that the second sum is also relatively small and thus the main contribution comes from the first sum. To do so we first write 
	\begin{equation}
		\label{eq:N_4(i)}
		N_4^{(i)}(B) 
		= M_4(B)
		\left( \sum_{a_4 \le B}1 - \sum_{\substack{a_4 \le B \\ -a_0a_4d \in \QQ^{*2}}}1 \right)
		+ O(B^2(\log B)^3),
	\end{equation}
	where
	\[
		M_4(B)
		=	\sum_{\substack{a_0, \dots, a_3 \le B \\ a_0a_1, a_0a_2, a_2a_3 \in \QQ^{*2}}} 1.
	\]

	We continue with the two sums over $a_4$ appearing above. Clearly, the first sum inside the brackets in \eqref{eq:N_4(i)} is $B + O(1)$. On the other hand, writing $a_4 = tu_4^2$ and $-a_0d = wu_5^2$ in the second sum with $t, w$ square-free shows that the the condition $-a_0a_4d \in \QQ^{*2}$ is fulfilled only if $t = w$. Hence one obtains
	\begin{equation}
		\label{eq:a_4 sum2}
		\sum_{\substack{a_4 \le B \\ -a_0a_4d \in \QQ^{*2}}} 1 
		\ll \sum_{u_4 \le \sqrt{B}} \sum_{\substack{t \le B/u_4^2 \\ t = w}} 1
		\ll B^{1/2}.
	\end{equation}

	As in the analysis of \eqref{eqn:d=0} the conditions $a_0a_1, a_0a_2, a_2a_3 \in \QQ^{*2}$ in $M_4(B)$ are detected by the factorisation
	\[
		a_0 = mk^2b_0^2, \quad
		a_1 = mk^2b_1^2, \quad
		a_2 = m\ell^2b_2^2, \quad
		a_3 = m\ell^2b_3^2,
	\]
	with $(k, \ell) = (b_0, b_1) = (b_2, b_3) = 1$. Thus
	\begin{equation}
		\label{eq:M_4}
		M_4(B) 
		= \sum_{m \le B}
		\sum_{\substack{k, \ell \le \sqrt{B/m}\\ (k, \ell) = 1}}
		\sum_{\substack{b_0, b_1 \le \sqrt{B/mk^2}\\ (b_0, b_1) = 1}}
		\sum_{\substack{b_2, b_3 \le \sqrt{B/m\ell^2}\\ (b_2, b_3) = 1}} 1.
	\end{equation}
	
	A standard argument in analytic number theory shows that for any real $X \ge 1$ the number of coprime $a, b \le X$ is $6X^2/\pi^2 + O(X \log X)$. We apply this  twice in \eqref{eq:M_4}, once for the sum over $b_2, b_3$ and once for the sum over $b_0, b_1$. This gives
	\[
		M_4(B)
		=\frac{36}{\pi^4}B^2 
		\sum_{m \le B} \frac{1}{m^2}
		\sum_{\substack{k, \ell \le \sqrt{B/m} \\ (k, \ell) = 1}} \frac{1}{k^2\ell^2}
		+ O(B^{3/2} (\log B)^2).
	\]
	All of the three sums appearing above are absolutely convergent. We first complete the sum over $\ell$, it equals $\zeta(2)\prod_{p \mid k}(1 - 1/p^2)$. The error in $M_4(B)$ coming from this completion is $O(B^{3/2})$. We do the same for the sum over $k$. The function inside the sum is multiplicative and hence this sum has an Euler product, it is $\prod_{p}(1 + 1/p^2)$. The error here is once again negligible compared to $B^{3/2} (\log B)^2$. Lastly, we complete the sum over $m$ to get
	\begin{equation}
		\label{eq:M_4 asymptotic}
		M_4(B)
		= \prod_p\(1 + \frac{1}{p^2}\) B^2
		+ O(B^{3/2} (\log B)^2).
	\end{equation}

	What is left is to combine \eqref{eq:N_4 N_4(i)}, \eqref{eq:N_4(i)}, \eqref{eq:a_4 sum2}, \eqref{eq:M_4 asymptotic} and to observe that the infinite product above equals $\zeta(2)/\zeta(4) = 15/\pi^2$. This gives
	\[
		N_4(B) 
		= \frac{60}{\pi^2} B^3
		+ O(B^{5/2}(\log B)^2),
	\]
	which completes the proof of Proposition~\ref{prop:Br4}.
\end{proof}

\section{Proof of Theorem~\ref{thm:Br}}
\label{sec:thm Br}
In order to prove Theorem~\ref{thm:Br} we first need to show that there are only a few surfaces in $\sF$ with trivial $\Br X_\bfa / \Br \QQ$. This is done in the next proposition.

\begin{proposition}
	\label{prop:N_1}
	We have
	\[
		N_1(B) \ll B^3 (\log B)^4.
	\]	 
\end{proposition}

\begin{proof}
	Recall that by Proposition~\ref{prop:BrXconic} we have $\Br X_\bfa / \Br \QQ$ trivial if and only if one of $-a_0a_4d$, $-a_1a_4d$, $a_2a_4d$, $a_3a_4d$ is a non-zero rational square. Let $N_1^{(0)}(B)$ be the number of those $X_\bfa \in \sF$ with $-a_0a_4d \in \QQ^{*2}$ and $|\bfa| \le B$. Define $N_1^{(1)}(B), N_1^{(2)}(B), N_1^{(3)}(B)$ in a similar fashion according the the conditions $-a_1a_4d$, $a_2a_4d$, $a_3a_4d \in \QQ^{*2}$. Forgetting the assumption on the existence of local points everywhere we clearly have
	\[
		N_1(B) \ll N_1^{(0)}(B) + N_1^{(1)}(B) + N_1^{(2)}(B) + N_1^{(3)}(B).
	\]  

	We shall now explain how to treat $N_1^{(0)}(B)$, the analysis of the other quantities being similar.  Letting
	\[
		a_0 = w_0v_0^2, \quad
		a_4 = w_4v_4^2, \quad
		-a_0a_1 + a_2a_3 = w_5v_5^2,
	\]
	with $w_0, w_4, w_5$ square-free allows us to see that $-a_0a_4d \in \QQ^{*2}$ is equivalent to $w_0w_4w_5$ being a non-zero rational square. Thus we must have $w_0 = st, w_4 = sw, w_5 = tw$ with $stw$ square-free. On the other hand, $-stv_0^2a_1 + a_2a_3 = twv_5^2$ implies that $t \mid a_2a_3$. Write $a_2 = t_2b_2$ and $a_3 = t_3b_3$, where $t = t_2t_3$ with $\mu^2(t_2t_3) = 1$. Thus
	\[
		b_2b_3 - a_1v_0^2s - v_5^2w = 0.
	\]

	Let $S, T_2, T_3, W, A_1, B_2, B_3, V_0, V_4, V_5 \gg 1$ run through the powers of 2. We shall also require them to satisfy the following conditions
	\[
		ST_2T_3V_0^2 \ll B, \quad
		A_1 \ll B, \quad
		T_2B_2 \ll B, \quad
		T_3B_3 \ll B, \quad
		SWV_4^2 \ll B
	\]
	and finally $WV_5^2 \ll B_2B_3 + A_1SV_0^2$. We break the the quantity $N_1^{(0)}(B)$ into sums $\Sigma = \Sigma(S, T_2, T_3, W, B_1, B_2, B_3, V_0, V_4, V_5)$ over the dyadic ranges $s \in (S/2, S]$, $t_2 \in (T_2/2, T_2]$ and so on. 

	What follows is an application of an upper bound of Heath-Brown \cite[Lem.~3]{HB84} which states that if $\bfalpha = (\alpha_1, \alpha_2, \alpha_3)$ is a integer vector with coprime coordinates, then
	\[
		\#\{\bfx \in \ZZ_{\text{prim}}^3 : |x_i| \le X_i, i = 1,2,3 \mbox{ and } \bfalpha \cdot \bfx = 0\}
		\ll 1 + \frac{X_1X_2X_3}{\max|\alpha_iX_i|}.
	\]
	We choose $\bfx = (a_1, b_2, w)$ and let $h = \gcd(a_1, b_2, w)$. Thus $X_1 = A_1$, $X_2 = B_2$, $X_3 = W$, $\alpha_1 = -sv_0^2$, $\alpha_2 = b_3$ and $\alpha_3 = v_5^2$. Let $\bfx' = \bfx/h$ and $\bfalpha' = \bfalpha/\gcd(sv_0^2, b_2, v_5^2)$. Applying Heath-Brown's bound for the number of $\bfx'$ with coordinates at most $X_i/h$ and satisfying $\bfx'\cdot \bfalpha' = 0$ now gives
	\[
			\#\{\bfx \in \ZZ^3 : |x_i| \le X_i \mbox{ and } \bfalpha \cdot \bfx = 0\}
			\ll \sum_{h \le \min\{A_1, B_2, W\}} \( 1 
			+ \frac{X_1X_2X_3 \gcd(sv_0^2, b_3, v_5^2)}{h^2\max|\alpha_i X_i|} \).
	\]
	
	It is clear that when we open the brackets the first sum is $O(W)$. On the other hand, the second sum over $h$ is convergent and the error coming from its tail is negligible. Thus summing over the remaining variables in the first sum and replacing $\max|\alpha_iX_i|$ by $(\alpha_1X_1 \alpha_2X_2 \alpha_3X_3)^{1/3}$ in the second sum gives
	\[
		\Sigma \ll ST_2T_3WB_3V_0V_4V_5 + T_2T_3(WA_1B_2)^{2/3}V_4\sum_{s, b_3, v_0, v_4} \frac{\gcd(sv_0^2, b_3, v_5^2)}{(sb_3v_0^2v_5^2)^{1/3}}.
	\]
	Running the same argument with $\bfx = (a_1, b_3, w)$ and then using  the elementary fact that $\min\{B_2, B_3\} \ll (B_2B_3)^{1/2}$ now gives
	\[
		\begin{split}
			\Sigma 
			\ll &ST_2T_3W(B_2B_3)^{1/2}V_0V_4V_5 
			+ T_2T_3(WA_1B_3)^{2/3}V_4\sum_{s, b_2, v_0, v_4} \frac{\gcd(sv_0^2, b_2, v_5^2)}{(sb_2v_0v_5)^{1/3}} \\
			&+ T_2T_3(WA_1B_2)^{2/3}V_4\sum_{s, b_3, v_0, v_4} \frac{\gcd(sv_0^2, b_3, v_5^2)}{(sb_3v_0v_5)^{1/3}}.
		\end{split}
	\]
	
	Our analysis so far showed that
	\[
		\begin{split}
			N_1^{(0)}(B)
			\ll &\sum_{S, T_2, T_3, W, A_1, B_2, B_3, V_0, V_4, V_5} ST_2T_3W(B_2B_3)^{1/2}V_0V_4V_5 \\
			&+ \sum_{S, T_2, T_3, W, A_1, B_2, B_3, V_0, V_4, V_5} T_2T_3(WA_1B_3)^{2/3}V_4
			\sum_{s, b_2, v_0, v_4} \frac{\gcd(sv_0^2, b_2, v_5^2)}{(sb_2v_0v_5)^{1/3}} \\
			&+ \sum_{S, T_2, T_3, W, A_1, B_2, B_3, V_0, V_4, V_5} T_2T_3(WA_1B_2)^{2/3}V_4
			\sum_{s, b_3, v_0, v_4} \frac{\gcd(sv_0^2, b_3, v_5^2)}{(sb_3v_0v_5)^{1/3}}.
		\end{split}
	\]
	Let $S_1$, $S_2$ and $S_3$ denote the first, the second and the third sum above, respectively.

	We claim that 
	\[	
		S_1  
		\ll B^3(\log B)^3.
	\]
	Indeed, recall the simple bound
	\begin{equation}
		\label{eqn:geometric}
		\sum_{A = 2^i \le X} A^{\theta}
		\ll
		\begin{cases}
			X^{\theta} &\mbox{if } \theta > 0, \\
			\log X &\mbox{if } \theta = 0, \\
			1 &\mbox{if } \theta < 0.
		\end{cases}
	\end{equation}
	With the arrangements made above we have guaranteed that
	\[
		U_0 \ll \(\frac{B}{ST_2T_3}\)^{1/2}, \quad
		U_4 \ll \(\frac{B}{SW}\)^{1/2}, \quad
		U_5 \ll \frac{B}{(T_2T_3W)^{1/2}}.
	\]
	Thus applying \eqref{eqn:geometric} to the sums over $B_2, B_3, U_0, U_4, U_5$ gives 
	\[
		S_1
		\ll B^3 \sum_{S, T_2, T_3, W, A_1}\frac{1}{(T_2T_3)^{1/2}}.
	\]
	We then apply \eqref{eqn:geometric} to the sums over the remaining variables to get
	\[ 
		S_1
		\ll B^3 (\log B)^3,
	\]
	which proves the claim.

	Finally, we claim that
	\[	
		S_2
		\ll B^3(\log B)^4, \quad
		S_3
		\ll B^3(\log B)^4.
	\]
	We will prove the above bound for $S_2$, the analysis for $S_3$ being similar. We begin by studying of the sum over $s, b_2, v_0, v_4$ appearing in $S_2$. Write
	\[
		s = ks', \quad
		b_2 = kmn^2c_2, \quad
		v_0 = mnu_0, \quad
		v_5 = kmnu_5.
	\]
	Then $\gcd(sv_0^2, b_2, v_5^2) = kmn^2$ and thus
	\[
		\sum_{s, b_2, v_0, v_4} \frac{\gcd(sv_0^2, b_2, v_5^2)}{(sb_2v_0v_5)^{1/3}}
		\ll \sum_{k, m, n, s', c_2, u_0, u_5} \frac{1}{(km^2s'c_2u_0^2u_5^2)^{1/3}},
	\]
	where the sum is over $ks' \ll S$, $kmn^2c_2 \ll B_2$, $mnu_0 \ll V_0$ and $kmnu_5 \ll V_5$. Summing over $s', u_0, c_2$ and $u_5$ then gives
	\[
		\sum_{s, b_2, v_0, v_4} \frac{\gcd(sv_0^2, b_2, v_5^2)}{(sb_2v_0v_5)^{1/3}}
		\ll (SB_2)^{2/3}(V_0V_5)^{1/3} \sum_{k, m, n} \frac{1}{(kmn)^2}.
	\]
	The sums over $k, m, n$ are convergent and their tails after completion have only negligible contribution. Thus
	\[
		S_2 
		\ll \sum_{S, T_2, T_3, W, A_1, B_2, B_3, V_0, V_4, V_5} S^{2/3}T_2T_3W^{2/3}(A_1B_2B_3)^{2/3}(V_0V_5)^{1/3}V_4.
	\]
	Proceeding in a similar fashion as in the study of $S_1$ now gives the claim which completes the proof.
\end{proof}

We continue with a lower bound for the quantity analysed in the previous proposition.

\begin{proposition}
	\label{prop:N_1_lower}
	We have
	\[
		N_1(B) \gg B^3.
	\]	 
\end{proposition}

\begin{proof}
	To show this we shall count the number of $\bfa \in \ZZ_{\text{prim}}^5$ coming from the following subfamily
	\[
		\sF' 
		= \{X_\bfa \in \sF \ : \ a_0a_1 - a_2a_3 \in \QQ^{*2} \mbox{ and } a_4 = -a_0\}.
	\]  
	Clearly, $(1: 0: 0: 0: 1) \in X_\bfa(\QQ)$ for each $X_\bfa \in \sF'$ and thus we need not worry about the existence of local points. Moreover, Proposition~\ref{prop:BrXconic} implies that all varieties in $\sF'$ have a trivial Brauer group and thus
	\[
		N_1(B) \gg \# \{X_\bfa \in \sF' \ : \ |\bfa| \le B\}.
	\]

	Let $N_1'(B)$ denote the cardinality of the set on the right hand side above. We then have
	\[
		\begin{split}
			N_1'(B)
			= \sum_{k \le B\sqrt{2} \ }
			\sum_{\substack{|a_0|, |a_1|, |a_2|, |a_3| \le B \\ a_0a_1 - a_2a_3 = k^2}} 1 
			\gg \sum_{k \le B/2 \ }
			\sum_{\substack{a_0^2 + a_1^2 + a_2^2 + a_3^2 \le B^2 \\ a_0a_1 - a_2a_3 = k^2}} 1. 
		\end{split}
	\]
	The inner sum in the far right hand side corresponds to the number of matrices in $M_{2}(\ZZ)$ with non-zero entries of height at most $B$ with respect to the standard Euclidean norm whose discriminant is equal to $k^2$. We shall apply \cite[Ex.~1.6]{DRS93} to this sum. Since in \cite[Ex.~1.6]{DRS93} quadruples with $a_i = 0$ are allowed we must first guarantee that the contribution from these $\bfa$ is negligible. Indeed, let $a_0 = 0$. Then $a_1$ can be chosen almost arbitrarily in the region $a_1^2 + a_2^2 + a_3^2 \le B^2$ and the choice of $a_2$ uniquely determined $a_3$ for each fixed $k$. Thus  
	\[
		\sum_{k \le B/2 \ } \sum_{\substack{a_1^2 + a_2^2 + a_3^2 \le B^2 \\ - a_2a_3 = k^2}} 1
		\ll B \sum_{k \le B/2} \tau(k^2) \ll B^{2 + \varepsilon},
	\]
	where $\varepsilon$ is arbitrarily small positive number and $\tau(k^2)$ denotes the number of divisors of $k^2$.	Applying \cite[Ex.~1.6]{DRS93} and taking the above into account now gives
	\[
		N_1'(B) 
		\gg B^2 \sum_{k \le B/2 \ } \sum_{d \mid k^2} \frac{1}{d}
		\gg B^3.
	\]
	This completes the proof.
\end{proof}

We are now in position to prove Theorem~\ref{thm:Br}. Recall that there are only three possibilities for $\Br X_\bfa/\Br \QQ$. With the notation set up earlier in the introduction we then have
\[
	N_2(B) 
	= \#\left\{X_\bfa \in \sF \ : \ |\bfa| \le B \mbox{ and } X_\bfa(\bfA_\QQ) \neq \emptyset \right\} - N_1(B) - N_4(B).
\]
On one hand, $N_4(B) = O(B^3)$ by Proposition~\ref{prop:Br4} and $B^3 \ll N_1(B) \ll B^3(\log B)^4$ by Propositions~\ref{prop:N_1} and \ref{prop:N_1_lower}. On the other hand, the remaining quantity above was studied in Theorem~\ref{thm:local solubility}. This proves Theorem~\ref{thm:Br}.
\qed

\section{Proof of Theorems~\ref{thm:BMO} and \ref{thm:WA}}
\label{sec:bmo}
We begin with the following simple lemma. 

\begin{lemma}
	\label{lem:surjectivity}
	Let $p > 7$ be a prime and let $a, b, c \in \Fp^*$. Then there exist $u_1, u_2 \in \Fp$ such that $u_1^2 + b \in \Fp^{*2}$ and $u_2^2 + b \in \Fp^* \setminus \Fp^{*2}$. Moreover, $u_1^2 + c$ and $u_2^2 + c$ are both units in $\Fp$ and $a(u_1^2 + b)(u_1^2 + c), a(u_2^2 + b)(u_2^2 + c) \in \Fp^{*2}$.
\end{lemma}

\begin{proof}
	Consider the projective curve $C$ defined over $\Fp$ by $v^2 = u^2 + bt^2$. This is a smooth quadric with an $\Fp$-rational point and thus it is isomorphic to $\PP^1$. Hence $\#C(\Fp) = p + 1$. If $-b \not\in \Fp^{*2}$, then all $\Fp$-points on $C$ satisfy $v \neq 0$. Alternatively, if $-b \in \Fp^{*2}$ there are two points in $C(\Fp)$ with $v = 0$. Finally, the divisor given by $t = 0$ consists of two $\Fp$-points. We conclude that
	\[
		\#\{(u, v) \in \Fp^2 \ : \ v^2 = u^2 + b \neq 0 \} = p - 2 - \(\frac{-b}{p}\).
	\]
	In particular, there exists $u_1 \in \Fp$ such that $u_1^2 + b \in \Fp^{*2}$ when the above quantity is positive. 

	To see the existence of $u_2$ as in the statement we apply a similar argument to show that
	\[
		\#\{(u, v) \in \Fp^2 \ : \  u^2 + b \neq 0 \} = p^2 - \(1 + \(\frac{-b}{p}\)\)p.
	\]
	Thus we have
	\[
		\#\(\{(u, v) \in \Fp^2 \ : \ u^2 + b \neq 0 \} 
		\setminus \{(u, v) \in \Fp^2 \ : \ v^2 = u^2 + b \neq 0 \}\)
		> p^2 - 3p + 1.
	\]
	It is clear that if $p > 3$, then both $p - 2 - \(\frac{-b}{p}\)$ and $p^2 - 3p + 1$ are positive. Moreover, if $p > 7$, then both quantities are at least $(p + 1)/2$ and thus we can choose $u_1$ and $u_2$ so that the assumptions of the statement are satisfied. This completes the proof of Lemma~\ref{lem:surjectivity}.
\end{proof}

Recall that the Brauer--Manin obstruction is known to be the only obstruction to the Hasse principle and weak approximation for $X_\bfa \in \sF$ since such surfaces are conic bundles with four degenerate geometric fibres \cite[Thm~2]{CT90}, \cite{Sal86}. A Brauer--Manin obstruction to the existence of rational points on $X_\bfa$ is present only if $\Br X_\bfa / \Br \QQ \simeq \ZZ/2\ZZ$ by Remark~\ref{rem:Br4}. Moreover, we can assume that $\Br X_\bfa / \Br \QQ \simeq \ZZ/2\ZZ$ since by Theorem~\ref{thm:Br} the number of surfaces with $\Br X_\bfa / \Br \QQ \not\simeq \ZZ/2\ZZ$ is negligible compared to the bounds we need to show in order to prove Theorems~\ref{thm:BMO} and \ref{thm:WA}. In this case, $\Br X_\bfa / \Br \QQ$ is generated by
\[
 	(a_0(x_0/x_2)^2 + a_2, -a_0a_4d)
 	=	((x_0/x_2)^2 + a_2/a_0, -a_0a_4d) \in \Br X_\bfa / \Br \QQ.
\]
Let $\alpha = ((x_0/x_2)^2 + a_2/a_0, -a_0a_4d)$ as an element of $\Br X_\bfa$. 

For a detailed background on the Brauer--Manin obstruction we refer the reader to \cite[Ch.~12]{CTS19} or \cite[Ch.~8]{P17}. What is important for us is that the Brauer--Manin set $X_\bfa(\bfA_\QQ)^{\Br X_\bfa}$ is determined by the adelic points on $X_\bfa$ for which the sum of local invariant maps $\sum_p \inv_p(\ev_\alpha(\bfx_p))$ vanish. Since the value of $\inv_p(\ev_\alpha(\bfx_p))$ is either $0$ or $1/2$, if there is a prime $p$ for which this map surjects on $\{0, 1/2\}$ we can modify the ad\`ele $(\bfx_p)$ at $p$ to get a point inside $X_\bfa(\bfA_\QQ)^{\Br X_\bfa}$. This shows that there is no Brauer--Manin obstruction to the Hasse principle in this case. It implies further that $X_\bfa(\bfA_\QQ)^{\Br X_\bfa}$ is a strict subset of $X_\bfa(\bfA_\QQ)$ and thus weak approximation fails.

Assume now that there is a prime $p > 7$ dividing $a_4$ to an odd power and such that $p$ does not divide $a_0a_1a_2a_3d$. On one hand, the invariant map $\inv_p(\ev_\alpha(\bfx_p))$ depends only on the Hilbert symbol $((x_0/x_2)^2 + a_2/a_0, -a_0a_4d)_p$ as seen by the following formula
\[
	\inv_p(\ev_\alpha(\bfx_p))
	= \frac{1 - ((x_0/x_2)^2 + a_2/a_0, -a_0a_4d)_p}{4}.
\]
In view of \eqref{eqn:conic bundle} Lemma~\ref{lem:surjectivity} applied with $u = x_0/x_2$, $a = a_0a_3$, $b = a_2/a_0$, $c = a_1/a_3$ and Hensel's lemma there are $\bfx_p', \bfx_p'' \in X_\bfa(\Qp)$ for which the Hilbert symbol takes both values $-1, 1$. Therefore, in order to obtain the upper bounds in Theorems~\ref{thm:BMO} and \ref{thm:WA} it suffices to count the number of $\bfa \in \ZZ_{\text{prim}}^5$ of height at most $B$ failing the above assumption.

It is clear that we only need to deal with divisibility conditions on the coordinates of $\bfa$. Since the sign of each $a_i$ is irrelevant to such conditions and we aim to prove only upper bounds we can freely assume that $a_i > 0$ for $i = 0, \dots, 4$. Let $k = \gcd(a_4, a_0)$, $\ell = \gcd(a_4/k, a_1)$, $m = \gcd(a_4/k\ell, a_2)$ and $n = \gcd(a_4/k\ell m, a_3)$. Then write
\begin{equation}
	\label{eqn:bfa N4}
	a_0 = k b_0, \quad
	a_1 = \ell b_1, \quad 
	a_2 = m b_2, \quad
	a_3 =	n b_3, \quad
	a_4 = 2^{v_2}3^{v_3}5^{v_5}7^{v_7} k \ell m n  r s b_4^2,
\end{equation}
with $v_p = \valp (a_4)$ for $p = 2, 3, 5, 7$. We have used $rs$ here for the square-free part of $a_4/2^{v_2}3^{v_3}5^{v_5}7^{v_7}k\ell mn$, where $r$ divides $k \ell mn$ while $s$ is coprime to $k \ell mn$. Our analysis above shows that if there is a Brauer--Manin obstruction to the Hasse principle, then for every prime $p \mid s$ we must have $k\ell b_0b_1 - mnb_2b_3 \equiv 0 \bmod p$. Since $s$ is square-free this condition is equivalent to  $k\ell b_0b_1 - mnb_2b_3  \equiv 0 \bmod s$. Therefore, the number of surfaces $X_\bfa \in \sF$ with a Brauer--Manin obstruction and such that $|\bfa| \le B$ is bounded from above by
\[
	 \#\{\bfa \in \ZZ_{\text{prim}}^5 : |a_i| \le B, \bfa \mbox{ satisfies } \eqref{eqn:bfa N4} \mbox{ and } k\ell b_0b_1 - mnb_2b_3 \equiv 0 \bmod s\}.
\] 

Let $N(B)$ denote this quantity. Forgetting the coprimality conditions except those between $s$ and $b_0, b_1, b_2, b_3, k, \ell, m , n$ now gives
\begin{equation}
	\label{eqn:N4 upper}
	N(B)
	\ll \sum_{\substack{k, \ell, m, n, r, s, v_i, b_0, b_1, b_2, b_4 \\ (s, b_0b_1b_2) = 1}} \mu^2(rs) 
	\sum_{\substack{b_3 \le B/n \\ b_3 \equiv k \ell b_0b_1(mnb_2)^{-1} \bmod s}} 1,
\end{equation}
where the summation is taken over $k b_0,  \ell b_1, m b_2, n b_3,  2^{v_2}3^{v_3}5^{v_5}7^{v_7} k \ell m n r s b_4^2 \le B$. We treat the inner sum in a standard way by splitting the interval $[1, B/n]$ into intervals of length $s$, which gives $\sum_{b_3} 1 = B/ns + O(1)$. Applying this in \eqref{eqn:N4 upper}, forgetting the remaining coprimality conditions and the squarefreeness of $rs$ and then summing over $b_4$ shows that 
\[
	N(B)
	\ll \sum_{k, \ell, m, n, r, s, v_i, b_0, b_1, b_2}
	\(\(\frac{B}{2^{v_2}3^{v_3}5^{v_5}7^{v_7}k\ell m n r s}\)^{1/2} + O(1)\)
	\(\frac{B}{n s} + O(1)\).
\]
We bound the remainder terms trivially to obtain
\[
	N(B)
	\ll B^{3/2}\sum_{k, \ell, m, n, r, s, v_i, b_0, b_1, b_2}
	\frac{1}{(2^{v_2}3^{v_3}5^{v_5}7^{v_7} k \ell m n^3 r s^3)^{1/2}}
	+ O(B^{4 + \varepsilon}),
\]
with $\varepsilon$ arbitrarily small.

The above sum over $s$ is clearly convergent. Thus summing over $s, b_0, b_1, b_2$ gives
\[
	\begin{split}
		N(B)\newcommand{\reqnomode}{\tagsleft@false}	
		&\ll B^{3/2}\sum_{k, \ell, m, n, r, v_i}
		\frac{1}{(2^{v_2}3^{v_3}5^{v_5}7^{v_7} k \ell m n^{3}r)^{1/2}}
		\(\frac{B}{k} + O(1)\)\(\frac{B}{\ell} + O(1)\)\(\frac{B}{m} + O(1)\) \\
		&+ O(B^{4 + \varepsilon}).
	\end{split}
\]
Since $r \mid k \ell m n$ it is cleat that $N(B) \ll B^{9/2}$. On the other hand, the quantities  appearing in Theorems~\ref{thm:BMO} and \ref{thm:WA} are $O(N(B))$. This proves Theorems~\ref{thm:BMO} and \ref{thm:WA}.
\qed

\bibliographystyle{amsalpha}{}
\bibliography{bibliography/references}
\end{document}